\documentclass[reqno,10pt,centertags,draft]{amsart}
\usepackage{amsmath,amsthm,amscd,amssymb,latexsym,upref}
\date{\today}

\newcommand{\bbN}{{\mathbb{N}}}
\newcommand{\bbR}{{\mathbb{R}}}

\newcommand{\bbC}{{\mathbb{C}}}

\newcommand{\cB}{{\mathcal B}}

\newcommand{\cH}{{\mathcal H}}

\newcommand{\cX}{{\mathcal X}}

\newcommand{\dott}{\,\cdot\,}
\newcommand{\no}{\notag}
\newcommand{\lb}{\label}
\newcommand{\f}{\frac}

\newcommand{\ol}{\overline}

\newcommand{\wti}{\widetilde}

\newcommand{\loc}{\text{\rm{loc}}}

\newcommand{\ran}{\text{\rm{ran}}}

\newcommand{\dom}{\text{\rm{dom}}}

\newcommand{\dist}{\text{\rm{dist}}}

\newcommand{\bi}{\bibitem}

\newcommand{\beq}{\begin{equation}}
\newcommand{\eeq}{\end{equation}}
\newcommand{\ba}{\begin{align}}
\newcommand{\ea}{\end{align}}

\newcommand{\tr}{\text{\rm{tr}}}

\newcommand{\abs}[1]{\lvert#1\rvert}

\renewcommand{\Im}{\text{\rm Im}}

\newcommand{\norm}[1]{\left\Vert#1\right\Vert}

\newcommand{\Om}{\Omega}
\newcommand{\dOm}{{\partial\Omega}}
\newcommand{\si}{\sigma}

\newcommand{\ga}{\gamma}

\newcommand{\eps}{\varepsilon}

\newcommand{\LOm}{L^2(\Om;d^nx)}
\newcommand{\LdOm}{L^2(\dOm;d^{n-1}\si)}

\allowdisplaybreaks \numberwithin{equation}{section}

\newtheorem{theorem}{Theorem}[section]

\newtheorem{lemma}[theorem]{Lemma}
\newtheorem{corollary}[theorem]{Corollary}
\newtheorem{hypothesis}[theorem]{Hypothesis}
\theoremstyle{definition}

\newtheorem{remark}[theorem]{Remark}

\begin{document}

\title[Dirichlet-to-Neumann Maps and Applications to Infinite Determinants]{On 
Dirichlet-to-Neumann Maps and Some Applications \\ to Modified Fredholm  Determinants}
\author[F.\ Gesztesy, M.\ Mitrea and M.\ Zinchenko]{Fritz Gesztesy, Marius Mitrea, and Maxim Zinchenko}
\address{Department of Mathematics,
University of Missouri, Columbia, MO 65211, USA}
\email{fritz@math.missouri.edu}
\urladdr{http://www.math.missouri.edu/personnel/faculty/gesztesyf.html}
\address{Department of Mathematics, University of
Missouri, Columbia, MO 65211, USA}
\email{marius@math.missouri.edu}
\urladdr{http://www.math.missouri.edu/personnel/faculty/mitream.html}
\address{Department of Mathematics,
California Institute of Technology, Pasadena, CA 91125, USA}
\email{maxim@caltech.edu}
\dedicatory{Dedicated with great pleasure to Boris Pavlov on the occasion of his 70th birthday}

\thanks{Based upon work partially supported by the US National Science
Foundation under Grant Nos.\ DMS-0405526, DMS-0400639, and 
FRG-0456306.}
\thanks{Appeared in {\it  Methods of Spectral Analysis in Mathematical Physics}, Conference on Operator Theory, Analysis and Mathematical Physics (OTAMP) 2006, Lund, Sweden, 
J.\ Janas, P.\ Kurasov, A.\ Laptev, S.\ Naboko, and G.\ Stolz (eds.), Operator Theory: Advances and Applications, Vol.\ {\bf 186}, Birkh\"auser, Basel, 2008, pp.\ 191--215.}
\date{\today}

\begin{abstract}
We consider Dirichlet-to-Neumann maps associated with (not necessarily self-adjoint) 
Schr\"odinger operators in $L^2(\Omega; d^n x)$, where $\Om\subset\bbR^n$, $n=2,3$, are open sets with a compact, nonempty boundary $\partial\Om$ satisfying certain regularity conditions. As an application we describe a reduction of a certain ratio of modified Fredholm perturbation determinants associated with operators in 
$L^2(\Om; d^n x)$ to modified Fredholm perturbation determinants associated with operators in $L^2(\partial\Om; d^{n-1}\sigma)$, $n=2,3$. This leads to a two- and three-dimensional extension of a variant of a celebrated formula due to Jost and Pais, which reduces the Fredholm perturbation determinant associated with a Schr\"odinger operator on the half-line $(0,\infty)$ to a simple Wronski determinant of appropriate distributional  solutions of the underlying Schr\"odinger equation. 
\end{abstract}

\maketitle

\section{Introduction}\label{s1}

To describe the original Fredholm determinant result due to Jost and Pais \cite{JP51}, we need a few preparations. Denoting by $H_{0,+}^D$ and $H_{0,+}^N$ the one-dimensional Dirichlet and Neumann Laplacians  in $L^2((0,\infty);dx)$, and assuming
\begin{equation}
V\in L^1((0,\infty);dx),   \lb{1.1}
\end{equation}
we introduce the perturbed Schr\"odinger operators
$H_{+}^D$ and $H_{+}^N$ in $L^2((0,\infty);dx)$ by
\begin{align}
&H_{+}^Df=-f''+Vf,  \no \\
&f\in \dom\big(H_{+}^D\big)=\{g\in L^2((0,\infty); dx) \,|\, g,g'
\in AC([0,R])
\text{ for all $R>0$}, \\
& \hspace*{4.8cm} g(0)=0, \, (-g''+Vg)\in L^2((0,\infty); dx)\}, \no \\
&H_{+}^Nf=-f''+Vf,  \no \\
&f\in \dom\big(H_{+}^N\big)=\{g\in L^2((0,\infty); dx) \,|\, g,g'
\in AC([0,R])
\text{ for all $R>0$}, \\
& \hspace*{4.7cm} g'(0)=0, \, (-g''+Vg)\in L^2((0,\infty); dx)\}. \no
\end{align}
Thus, $H_{+}^D$ and $H_{+}^N$ are self-adjoint if and only if $V$ is
real-valued, but since the latter restriction plays no special role in our results, we
will not assume real-valuedness of $V$ throughout this paper.

A fundamental system of solutions $\phi_+^D(z,\cdot)$,
$\theta_+^D(z,\cdot)$, and the Jost solution $f_+(z,\cdot)$ of
\begin{equation}
-\psi''(z,x)+V\psi(z,x)=z\psi(z,x), \quad z\in\bbC\backslash\{0\}, \;
x\geq 0,   \lb{1.4}
\end{equation}
are then introduced via the standard Volterra integral equations
\begin{align}
\phi_+^D(z,x)&=z^{-1/2}\sin(z^{1/2}x)+\int_0^x dx' \, z^{-1/2}\sin(z^{1/2}(x-x'))
V(x')\phi_+^D(z,x'), \\
\theta_+^D(z,x)&=\cos(z^{1/2}x)+\int_0^x dx' \, z^{-1/2}\sin(z^{1/2}(x-x'))
V(x')\theta_+^D (z,x'), \\
f_+(z,x)&=e^{iz^{1/2}x}-\int_x^\infty dx' \,
z^{-1/2}\sin(z^{1/2}(x-x')) V(x')f_+(z,x'),  \lb{1.7} \\
&\hspace*{3.85cm} z\in\bbC\backslash\{0\}, \; \Im(z^{1/2})\geq 0, \;
x\geq 0.  \no
\end{align}

In addition, we introduce
\begin{equation}
u=\exp(i\arg(V))\abs{V}^{1/2}, \quad v=\abs{V}^{1/2}, \, \text{ so
that } \, V=u\, v,
\end{equation}
and denote by $I_+$ the identity operator in $L^2((0,\infty); dx)$. Moreover,
we denote by
\begin{equation}
W(f,g)(x)=f(x)g'(x)-f'(x)g(x), \quad x \geq 0,
\end{equation}
the Wronskian of $f$ and $g$, where $f,g \in C^1([0,\infty))$. We
also use the standard convention to abbreviate (with a slight abuse of notation)
the operator of multiplication in $L^2((0,\infty);dx)$ by an element 
$f\in L^1_{\loc}((0,\infty);dx)$ (and similarly in the higher-dimensional
context with $(0,\infty)$ replaced by an appropriate open set 
$\Omega\subset\bbR^n$ later) by the same symbol $f$ (rather than $M_f$, etc.). 
For additional notational conventions we refer to the paragraph at the end of this introduction.

Then, the following results hold (with $\cB_1(\cdot)$ abbreviating the ideal of trace class operators):

\begin{theorem} \lb{t1.1}
Assume $V\in L^1((0,\infty);dx)$ and let $z\in\bbC\backslash [0,\infty)$
with $\Im(z^{1/2})>0$. Then,
\begin{equation}
\ol{u\big(H_{0,+}^D-z I_+\big)^{-1}v}, \,
\ol{u\big(H_{0,+}^N-z I_+\big)^{-1}v} \in \cB_1(L^2((0,\infty);dx))
\end{equation}
and
\begin{align}
\det\Big(I_+ +\ol{u\big(H_{0, +}^D-z I_+\big)^{-1}v}\,\Big) &=
1+z^{-1/2}\int_0^\infty dx\, \sin(z^{1/2}x)V(x)f_+(z,x)   \no \\
&= W(f_+(z,\cdot),\phi_+^D(z,\cdot)) = f_+(z,0),    \lb{1.11}  \\
\det\Big(I_+ +\ol{u\big(H_{0, +}^N-z I_+\big)^{-1}v}\,\Big)
&= 1+ i z^{-1/2} \int_0^\infty
dx\, \cos(z^{1/2}x)V(x)f_+(z,x) \no  \\
&= - \frac{W(f_+(z,\cdot),\theta_+^D (z,\cdot))}{i z^{1/2}} =
\frac{f_+'(z,0)}{i z^{1/2}}.    \lb{1.12}
\end{align}
\end{theorem}

Equation \eqref{1.11} is the modern formulation of the celebrated result due to Jost and
Pais \cite{JP51}. Performing calculations similar to Section 4 in \cite{GM03} for the
pair of operators $H_{0,+}^N$ and $H_+^N$, one obtains the analogous result
\eqref{1.12}. 

We emphasize that \eqref{1.11} and \eqref{1.12} exhibit a spectacular
reduction of a Fredholm determinant, that is, an infinite determinant (actually, a symmetrized perturbation determinant), associated with the trace class 
Birmann--Schwinger kernel of a one-dimensional Schr\"odinger operator on the half-line $(0,\infty)$, to a simple Wronski determinant of $\bbC$-valued distributional solutions of
\eqref{1.4}. This fact goes back to Jost and Pais \cite{JP51} (see
also \cite{GM03}, \cite{Ne72}, \cite{Ne80}, \cite[Sect.\ 12.1.2]{Ne02},
\cite{Si00}, \cite[Proposition 5.7]{Si05}, and the extensive literature
cited in these references). The principal aim of this paper is to explore
the extent to which this fact may generalize to higher dimensions. While a 
direct generalization of \eqref{1.11},
\eqref{1.12} appears to be difficult, we will next derive a formula for
the ratio of such determinants which indeed permits a natural extension to
higher dimensions.

For this purpose we introduce the boundary trace operators
$\ga_D$ (Dirichlet trace) and $\ga_N$ (Neumann trace) which, in the
current one-dimensional half-line situation, are just the functionals,
\begin{equation}
\ga_D \colon \begin{cases} C([0,\infty)) \to \bbC, \\
\hspace*{1.3cm} g \mapsto g(0), \end{cases}  \quad
\ga_N \colon \begin{cases}C^1([0,\infty)) \to \bbC,  \\
\hspace*{1.43cm} h \mapsto   - h'(0). \end{cases}
\end{equation}
In addition, we denote by $m_{0,+}^D$, $m_+^D$, $m_{0,+}^N$, and $m_+^N$
the Weyl--Titchmarsh $m$-functions corresponding to $H_{0,+}^D$,
$H_{+}^D$, $H_{0,+}^N$, and $H_{+}^N$, respectively, that is,
\begin{align}
m_{0,+}^D(z) &= i z^{1/2},  \qquad m_{0,+}^N (z)= -\frac{1}{m_{0,+}^D(z)}
= i z^{-1/2},  \lb{1.14} \\
m_{+}^D(z) &= \frac{f_+'(z,0)}{f_+(z,0)}, \quad m_{+}^N (z)=
-\frac{1}{m_{+}^D(z)} = -\frac{f_+(z,0)}{f_+'(z,0)}.  \lb{1.15}
\end{align}
In the case where $V$ is real-valued, we briefly recall the spectral theoretic significance of $m_{+}^D$: It is a Herglotz function (i.e., it maps the open complex upper half-plane $\bbC_+$ analytically into itself) and the measure $d\rho^D_+$ in its Herglotz representation is then the spectral measure of the operator $H_{+}^D$ and hence encodes all spectral information of 
$H_{+}^D$. Similarly, $m_{+}^D$ also encodes all spectral information of $H_{+}^N$ since $-1/m_{+}^D = m_{+}^N$ is also a Herglotz function and the measure
$d\rho^N_+$ in its Herglotz representation represents the spectral measure of the operator $H_{+}^N$. In particular, $d\rho^D_+$ (respectively, $d\rho^N_+$) uniquely determine $V$ a.e.\ on $(0,\infty)$ by the inverse spectral approach of Gelfand and Levitan \cite{GL55} or Simon \cite{Si99}, \cite{GS00} (see also Remling \cite{Re03} and Section 6 in the survey \cite{Ge07}).

Then we obtain the following result for the ratio of the perturbation determinants in
\eqref{1.11} and \eqref{1.12}:

\begin{theorem} \lb{t1.2}
Assume $V\in L^1((0,\infty);dx)$ and let $z\in\bbC\backslash\si(H_+^D)$
with $\Im(z^{1/2})>0$. Then,
\begin{align}
& \frac{\det\Big(I_+ +\ol{u\big(H_{0, +}^N-z I_+\big)^{-1} v}\,\Big)}
{\det\Big(I_+ +\ol{u\big(H_{0, +}^D-z I_+\big)^{-1}v}\,\Big)}
\no \\
&\quad = 1 - \big(\ol{\ga_N(H_+^D-z I_+)^{-1}V
\big[\ga_D(H_{0,+}^N-\ol{z}I_+)^{-1}\big]^*}\,\big)  \lb{1.16}  \\
& \quad = \f{W(f_+(z),\phi_+^N(z))}{i z^{1/2}W(f_+(z),\phi_+^D(z))} =
\f{f'_+(z,0)}{i z^{1/2}f_+(z,0)} = \f{m_+^D(z)}{m_{0,+}^D(z)} =
\f{m_{0,+}^N(z)}{m_+^N(z)}.   \lb{1.17}
\end{align}
\end{theorem}

The proper multi-dimensional generalization to Schr\"odinger operators in 
$L^2(\Om;d^n x)$ corresponding to an open set $\Om \subset \bbR^n$ with compact, nonempty boundary $\partial\Om$ then involves the operator-valued generalization of the Weyl--Titchmarsh function $m_+^D(z)$, the Dirichlet-to-Neumann map denoted by 
$M_{\Om}^D(z)$. In particular, we will derive the following multi-dimensional extension of \eqref{1.16} and \eqref{1.17} in Section \ref{s4}:

\begin{theorem} \lb{t1.3}
Assume Hypothesis \ref{h2.6} and let
$z\in\bbC\big\backslash\big(\si\big(H_{\Om}^D\big)\cup \si\big(H_{0,\Om}^D\big) \cup
\si\big(H_{0,\Om}^N\big)\big)$. Then,
\begin{align}
& \frac{\det{}_2\Big(I_{\Om}+\ol{u\big(H_{0,\Om}^N-zI_{\Om}\big)^{-1}v}\,\Big)}
{\det{}_2\Big(I_{\Om}+\ol{u\big(H_{0,\Om}^D-zI_{\Om}\big)^{-1}v}\,\Big)} \no \\
& \quad = \det{}_2\Big(I_{\dOm} - \ol{\ga_N\big(H_{\Om}^D-zI_{\Om}\big)^{-1} V
\big[\ga_D(H_{0,\Om}^N-\ol{z}I_{\Om})^{-1}\big]^*}\,\Big)
e^{\tr(T_2(z))}   \lb{1.18}  \\
& \quad =
\det{}_2\big(M_{\Om}^{D}(z)M_{0,\Om}^{D}(z)^{-1}\big)
e^{\tr(T_2(z))}.   \lb{1.19}
\end{align}
\end{theorem}
Here, ${\det}_2(\cdot)$ denotes the modified Fredholm determinant in
connection with Hilbert--Schmidt perturbations of the identity, $T_2(z)$ is given by
\begin{equation}
T_2(z)=\ol{\ga_N\big(H_{0,\Om}^D-zI_{\Om}\big)^{-1} V
\big(H_{\Om}^D-zI_{\Om}\big)^{-1} V
\big[\ga_D \big(H_{0,\Om}^N-\ol{z}I_{\Om}\big)^{-1}\big]^*},
\end{equation}
and $I_{\Om}$ and $I_{\partial\Om}$ represent the identity operators in
$L^2(\Om; d^n x)$ and  $L^2(\partial\Om; d^{n-1} \sigma)$, respectively (with
$d^{n-1}\sigma$ the surface measure on $\partial\Om$). 

For pertinent comments on the principal reduction of (a ratio of) modified Fredholm determinants associated with operators in $L^2(\Om; d^n x)$ on the left-hand side of \eqref{1.18} to a modified Fredholm determinant associated with operators in
$L^2(\partial\Om; d^{n-1} \sigma)$ on the right-hand side of \eqref{1.18} and especially, in \eqref{1.19}, we refer to Section \ref{s4}.

Finally, we briefly list some of the notational conventions used throughout this
paper. Let $T$ be a linear operator
mapping (a subspace of) a Banach space into another, with
$\dom(T)$ and $\ran(T)$ denoting the domain and range of $T$.
The closure of a closable operator $S$ is denoted by $\ol S$.
The kernel (null space) of $T$ is denoted by $\ker(T)$. The
spectrum and resolvent set of a closed linear operator in 
a separable complex Hilbert space $\cH$ (with scalar product  
denoted by $(\cdot,\cdot)_{\cH}$, assumed to be linear in the second factor) 
will be denoted by $\sigma(\cdot)$ and $\rho(\cdot)$. The Banach
spaces of bounded and compact linear operators in $\cH$ are
denoted by $\cB(\cH)$ and $\cB_\infty(\cH)$, respectively.
Similarly, the Schatten--von Neumann (trace) ideals will
subsequently be denoted by $\cB_p(\cH)$, $p\in\bbN$. Analogous
notation $\cB(\cH_1,\cH_2)$, $\cB_\infty (\cH_1,\cH_2)$, etc.,
will be used for bounded, compact, etc., operators between two
Hilbert spaces $\cH_1$ and $\cH_2$. In addition, $\tr(T)$
denotes the trace of a trace class operator $T\in\cB_1(\cH)$ and
$\det_{p}(I_{\cH}+S)$ represents the (modified) Fredholm
determinant associated with an operator $S\in\cB_p(\cH)$,
$p\in\bbN$ (for $p=1$ we omit the subscript $1$). Moreover,
$\cX_1 \hookrightarrow \cX_2$ denotes the continuous embedding
of the Banach space $\cX_1$ into the Banach space $\cX_2$.

For general references on the theory of modified Fredholm determinants we 
refer, for instance, to \cite[Sect.\ XI.9]{DS88}, \cite[Chs.\  IX, XI]{GGK00}, 
\cite[Sect.\ IV.2]{GK69}, \cite{Si77}, and \cite[Ch.\ 9]{Si05}.

\section{Schr\"odinger Operators with Dirichlet and Neumann boundary conditions}
\label{s2}

In this section we primarily focus on various properties of Dirichlet,
$H^D_{0,\Om}$, and Neumann, $H^N_{0,\Om}$, Laplacians in
$L^2(\Om;d^n x)$ associated with open sets $\Omega\subset \bbR^n$, $n=2,3$, introduced in Hypothesis \ref{h2.1} below. In particular, we study mapping properties of $\big(H^{D,N}_{0,\Om}-zI_{\Om}\big)^{-q}$, $q\in [0,1]$, 
($I_{\Om}$ the identity operator in $L^2(\Om; d^n x)$) and trace ideal properties of the maps $f \big(H^{D,N}_{0,\Om}-zI_{\Om}\big)^{-q}$, $f\in L^p(\Om; d^n x)$, for appropriate $p\geq 2$, and
$\gamma_N \big(H^{D}_{0,\Om}-zI_{\Om}\big)^{-r}$, and
$\gamma_D \big(H^{N}_{0,\Om}-zI_{\Om}\big)^{-s}$, for appropriate
$r>3/4$, $s>1/4$, with $\gamma_N$ and $\gamma_D$ being the Neumann and Dirichlet boundary trace operators defined in \eqref{2.2} and \eqref{2.3}.

At the end of this section we then introduce the Dirichlet and Neumann
Schr\"odinger operators $H^D_{\Om}$ and $H^N_{\Om}$ in
$L^2(\Om;d^n x)$, that is, perturbations of the Dirichlet and Neumann Laplacians $H^D_{0,\Om}$ and $H^N_{0,\Om}$ by a potential $V$ satisfying
Hypothesis \ref{h2.6}.

We start with introducing our assumptions on the set $\Omega$:

\begin{hypothesis} \lb{h2.1}
Let $n=2,3$ and assume that $\Omega\subset{\bbR}^n$ is an open
set with a compact, nonempty boundary $\partial\Omega$. In addition,
we assume that one of the following three conditions holds: \\
$(i)$ \, $\Omega$ is of class $C^{1,r}$ for some $1/2 < r <1$; \\
$(ii)$  \hspace*{.0001pt} $\Omega$ is convex; \\
$(iii)$ $\Omega$ is a Lipschitz domain satisfying a {\it uniform exterior ball condition}. 
\end{hypothesis}

The class of domains described in Hypothesis \ref{h2.1} is a subclass of all Lipschitz domains with compact nonempty boundary. We also note that while $\dOm$ is assumed to be compact, $\Om$ may be unbounded 
(e.g., an exterior domain) in connection with conditions $(i)$ or $(iii)$. For more details in this context, in particular, for the precise definition of the uniform exterior ball condition, we refer to \cite[App.\ A]{GLMZ05} and \cite[App.\ A]{GMZ06} (and the references cited therein, such as \cite[Ch.\ 1]{Gr85}, \cite{GI75}, \cite{JK81}, \cite{Ka64}, \cite[Ch.\ 3]{Mc00}, 
\cite{Mi96}, \cite{Mi01}, \cite[p.\ 189]{St70}, \cite{Ta65}, \cite{Tr02}, and 
\cite[Sect.\ I.4.2]{Wl87}).

First, we introduce the boundary trace operator $\ga_D^0$
(Dirichlet trace) by
\begin{equation}
\ga_D^0\colon C(\ol{\Om})\to C(\dOm), \quad \ga_D^0 u = u|_\dOm .
\end{equation}
Then there exists a bounded, linear operator $\gamma_D$ (cf.
\cite[Theorem 3.38]{Mc00}),
\begin{align}
\begin{split}
& \ga_D\colon H^{s}(\Om)\to H^{s-(1/2)}(\dOm) \hookrightarrow \LdOm,
\quad 1/2<s<3/2, \lb{2.2}  \\
& \ga_D\colon H^{3/2}(\Om)\to H^{1-\varepsilon}(\dOm) \hookrightarrow \LdOm,
\quad \varepsilon \in (0,1), 
\end{split}
\end{align} 
whose action is compatible with that of $\ga_D^0$. That is, the two
Dirichlet trace  operators coincide on the intersection of their
domains. We recall that $d^{n-1}\sigma$ denotes the surface measure on
$\dOm$ and we refer to \cite[App.\ A]{GLMZ05} for our notation in
connection with Sobolev spaces (see also \cite[Ch.\ 3]{Mc00}, \cite{Tr02}, and 
\cite[Sect.\ I.4.2]{Wl87}).

Next, we introduce the operator $\ga_N$ (Neumann trace) by
\begin{align}
\ga_N = \nu\cdot\ga_D\nabla \colon H^{s+1}(\Om)\to \LdOm, \quad
1/2<s<3/2, \lb{2.3}
\end{align}
where $\nu$ denotes the outward pointing normal unit vector to
$\partial\Om$. It follows from \eqref{2.2} that $\ga_N$ is also a
bounded operator.

Given Hypothesis \ref{h2.1}, we introduce the self-adjoint and nonnegative Dirichlet and
Neumann Laplacians $H_{0,\Om}^D$ and $H_{0,\Om}^N$ associated
with the domain $\Om$ as follows,
\begin{align}
H^D_{0,\Om} = -\Delta, \quad \dom\big(H^D_{0,\Om}\big) = \{u\in H^{2}(\Om)
\,|\, \ga_D u = 0\}, \lb{2.4}
\\
H^N_{0,\Om} = -\Delta, \quad \dom\big(H^N_{0,\Om}\big) = \{u\in H^{2}(\Om)
\,|\, \ga_N u = 0\}. \lb{2.5}
\end{align}
A detailed discussion of $H^D_{0,\Om}$ and $H^N_{0,\Om}$ is provided in 
\cite[App.\ A]{GLMZ05} (cf.\ also \cite[Sects. X.III.14, X.III.15]{RS78}).

\begin{lemma} \lb{l2.2}
Assume Hypothesis \ref{h2.1}. Then the operators $H_{0,\Om}^D$
and $H_{0,\Om}^N$ introduced in \eqref{2.4} and \eqref{2.5} are
nonnegative and self-adjoint in $\LOm$ and the following boundedness 
properties hold for all $q\in [0,1]$ and $z\in\bbC\backslash[0,\infty)$, 
\begin{align}
\big(H_{0,\Om}^D-zI_{\Om}\big)^{-q},\, \big(H_{0,\Om}^N-zI_{\Om}\big)^{-q}
\in\cB\big(\LOm,H^{2q}(\Om)\big). \lb{2.6}
\end{align}
\end{lemma}

The fractional powers in \eqref{2.6} (and in subsequent analogous
cases) are defined via the functional calculus implied by the
spectral theorem for self-adjoint operators.

As explained in \cite[Lemma A.2]{GLMZ05} (based on results in \cite{JK95}, 
\cite[Thm.\ 4.4, App.\ B]{Mc00}, \cite{MMT01}, \cite{MT00}, \cite[Chs.\ 1, 2]{Ta91}, 
\cite[Props.\ 4.5, 7.9]{Ta96}, \cite[Sect.\ 1.3, Thm.\ 1.18.10, Rem.\ 4.3.1.2]{Tr78}, \cite{Ve84}), the key ingredients in proving 
Lemma \ref{l2.2} are the inclusions
\begin{equation}
\dom\big(H_{0,\Om}^D\big) \subset H^2(\Om), \quad
\dom\big(H^N_{0,\Om}\big) \subset H^2(\Om)
\end{equation}
and real interpolation methods. 

The next result is a slight extension of \cite[Lemma 6.8]{GLMZ05} and provides an explicit discussion of the $z$-dependence of the constant $c$ appearing in estimate (6.48) of \cite{GLMZ05}. For a proof we refer to \cite{GMZ06}.

\begin{lemma} [\cite{GMZ06}] \lb{l2.3}
Assume Hypothesis \ref{h2.1} and let $2\leq p$, $n/(2p)<q\leq1$,
$f\in L^p(\Om;d^nx)$, and $z\in\bbC\backslash[0,\infty)$. Then,
\begin{align} \lb{2.7}
f\big(H_{0,\Om}^D-zI_{\Om}\big)^{-q}, \, f\big(H_{0,\Om}^N-zI_{\Om}\big)^{-q}
\in\cB_p\big(\LOm\big),
\end{align}
and for some $c>0$ $($independent of $z$ and $f$\,$)$
\begin{align}
\begin{split}
&\big\| f\big(H_{0,\Om}^D-zI_{\Om}\big)^{-q}\big\|_{\cB_p(\LOm)}^2
\\
& \qquad \leq
c\bigg(1+\frac{\abs{z}^{2q}+1}{\dist\big(z,\si\big(H_{0,\Om}^D\big)\big)^{2q}}\bigg)
\|(\abs{\dott}^2-z)^{-q}\|_{L^p(\bbR^n;d^nx)}^2
\|f\|_{L^p(\Om;d^nx)}^2,
\\
&\big\| f\big(H_{0,\Om}^N-zI_{\Om}\big)^{-q}\big\|_{\cB_p(\LOm)}^2
\\
& \qquad \leq
c\bigg(1+\frac{\abs{z}^{2q}+1}{\dist\big(z,\si\big(H_{0,\Om}^N\big)\big)^{2q}}\bigg)
\|(\abs{\dott}^2-z)^{-q}\|_{L^p(\bbR^n;d^nx)}^2
\|f\|_{L^p(\Om;d^nx)}^2.
\end{split} \lb{2.8}
\end{align}
\end{lemma}

(Here, in obvious notation, $(\abs{\dott}^2-z)^{-q}$ denotes the function 
$(\abs{x}^2-z)^{-q}$, $x\in\bbR^n$.)

Next we recall certain boundedness  properties of powers of the resolvents of Dirichlet and Neumann Laplacians multiplied by the Neumann and Dirichlet boundary trace operators, respectively:

\begin{lemma} \lb{l2.4}
Assume Hypothesis \ref{h2.1} and let $\eps > 0$,
$z\in\bbC\backslash[0,\infty)$. Then,
\begin{align}
\ga_N\big(H_{0,\Om}^D-zI_{\Om}\big)^{-\frac{3+\eps}{4}},
\ga_D\big(H_{0,\Om}^N-zI_{\Om}\big)^{-\frac{1+\eps}{4}} \in
\cB\big(\LOm,\LdOm\big).  \lb{2.22}
\end{align}
\end{lemma}

As in \cite[Lemma 6.9]{GLMZ05}, Lemma \ref{l2.4} follows from Lemma \ref{l2.2}
and from \eqref{2.2} and \eqref{2.3}.

\begin{corollary} \lb{c2.5}
Assume Hypothesis \ref{h2.1} and let $f_1\in L^{p_1}(\Om;d^nx)$,
$p_1\geq 2$, $p_1>2n/3$, $f_2\in L^{p_2}(\Om;d^nx)$, $p_2 > 2n$,
and $z\in\bbC\backslash[0,\infty)$. Then, denoting by $f_1$ and
$f_2$ the operators of multiplication by functions $f_1$ and
$f_2$ in $\LOm$, respectively, one has
\begin{align}
\ol{\ga_D\big(H_{0,\Om}^N-zI_{\Om}\big)^{-1}f_1} &\in
\cB_{p_1}\big(L^2(\Om;d^nx),L^2(\dOm;d^{n-1}\si)\big), \lb{2.25}
\\
\ol{\ga_N\big(H_{0,\Om}^D-zI_{\Om}\big)^{-1}f_2} &\in
\cB_{p_2}\big(L^2(\Om;d^nx),L^2(\dOm;d^{n-1}\si)\big) \lb{2.26}
\end{align}
and for some $c_j(z)>0$ $($independent of $f_j$$)$, $j=1,2$,
\begin{align}
\Big\| \, \ol{\ga_D\big(H_{0,\Om}^N-zI_{\Om}\big)^{-1}f_1}\, \Big\|_
{\cB_{p_1}(L^2(\Om;d^nx),L^2(\dOm;d^{n-1}\si))} &\leq c_1(z)
\norm{f_1}_{L^{p_1}(\Om;d^nx)}, \lb{2.27}
\\
\Big\| \, \ol{\ga_N\big(H_{0,\Om}^D-zI_{\Om}\big)^{-1}f_2}\, \Big\|_
{\cB_{p_2}(L^2(\Om;d^nx),L^2(\dOm;d^{n-1}\si))} &\leq c_2(z)
\norm{f_2}_{L^{p_2}(\Om;d^nx)}. \lb{2.28}
\end{align}
\end{corollary}

As in \cite[Corollary 6.10]{GLMZ05}, Corollary \ref{c2.5} follows from Lemmas \ref{l2.3} and \ref{l2.4}.

Finally, we turn to our assumptions on the potential $V$ and the corresponding definition of Dirichlet and Neumann Schr\"odinger operators $H^D_{\Om}$ and $H^N_{\Om}$ in $L^2(\Om; d^n x)$:

\begin{hypothesis} \lb{h2.6}
Suppose that $\Om$ satisfies Hypothesis \ref{h2.1}  and assume that
$V\in L^p(\Om;d^nx)$ for some $p$ satisfying $4/3
< p \leq 2$, in the case $n=2$, and $3/2 < p \leq 2$, in the case
$n=3$.
\end{hypothesis}

Assuming Hypothesis \ref{h2.6}, we next introduce the perturbed operators
$H_{\Om}^D$ and $H_{\Om}^N$ in $\LOm$ by alluding to abstract perturbation results 
due to Kato \cite{Ka66} (see also Konno and Kuroda \cite{KK66}) as 
summarized in \cite[Sect.\ 2]{GLMZ05}: Let $V$, $u$, and $v$ denote
the operators of multiplication by functions $V$,
$u=\exp(i\arg(V))\abs{V}^{1/2}$, and $v=\abs{V}^{1/2}$ in
$L^{2}(\Om;d^nx)$, respectively, such that
\begin{equation}
V=uv.   \lb{2.29}
\end{equation}
Since $u,v\in L^{2p}(\Om;d^nx)$, Lemma \ref{l2.3} yields
\begin{align}
u\big(H_{0,\Om}^D-zI_{\Om}\big)^{-1/2}, \,
\ol{\big(H_{0,\Om}^D-zI_{\Om}\big)^{-1/2}v} &\in\cB_{2p}\big(\LOm\big),
\quad z\in\bbC\backslash [0,\infty), \lb{2.31}
\\
u\big(H_{0,\Om}^N-zI_{\Om}\big)^{-1/2}, \,
\ol{\big(H_{0,\Om}^N-zI_{\Om}\big)^{-1/2}v} &\in\cB_{2p}\big(\LOm\big),
\quad z\in\bbC\backslash [0,\infty), \lb{2.32}
\end{align}
and hence, in particular,
\begin{align}
& \dom(u)=\dom(v) \supseteq  \dom\Big(\big(H^N_{0,\Om}\big)^{1/2}\Big) = 
H^{1}(\Om) \supset H^{2}(\Om) \supset \dom\big(H^N_{0,\Om}\big), \\
& \dom(u)=\dom(v) \supseteq  H^{1}(\Om) \supset H^{1}_0(\Om) 
= \dom\Big(\big(H^D_{0,\Om}\big)^{1/2}\Big) 
\supset \dom\big(H^D_{0,\Om}\big).
\end{align}
Moreover, \eqref{2.31} and \eqref{2.32} imply
\begin{equation}
\ol{u\big(H_{0,\Om}^D-zI_{\Om}\big)^{-1}v}, \,
\ol{u\big(H_{0,\Om}^N-zI_{\Om}\big)^{-1}v}
\in\cB_p\big(\LOm\big)\subset\cB_2\big(\LOm\big), \quad
z\in\bbC\backslash [0,\infty).  \lb{2.35}
\end{equation}
Utilizing \eqref{2.8} in Lemma \ref{l2.3}
with $-z>0$ sufficiently large, such that the $\cB_{2p}$-norms
of the operators in \eqref{2.31} and \eqref{2.32} are less than
1, one concludes that the Hilbert--Schmidt norms of the operators in
\eqref{2.35} are less than 1. Thus, applying \cite[Thm.\ 2.3]{GLMZ05}, one
obtains the densely defined, closed operators $H_{\Om}^D$ and
$H_{\Om}^N$ (which are extensions of $H_{0,\Om}^D+V$ defined on
$\dom\big(H_{0,\Om}^D\big)\cap\dom(V)$ and $H_{0,\Om}^N+V$ defined on
$\dom\big(H_{0,\Om}^N\big)\cap\dom(V)$, respectively). In particular, the
resolvent of $H_{\Om}^D$ (respectively, $H_{\Om}^N$) is explicitly given by
\begin{align}
\big(H^D_{\Om}-zI_\Om\big)^{-1} &= \big(H^D_{0,\Om}-zI_\Om\big)^{-1} -
\big(H^D_{0,\Om}-zI_\Om\big)^{-1}v
\Big[I_\Om+\ol{u\big(H^D_{0,\Om}-zI_\Om\big)^{-1}v}\,\Big]^{-1}
u\big(H^D_{0,\Om}-zI_\Om\big)^{-1},   \no \\
& \hspace*{9.1cm}  z\in\bbC\backslash\sigma\big(H^D_{\Om}\big), \lb{2.20}  \\
\big(H^N_{\Om}-zI_\Om\big)^{-1} &= \big(H^N_{0,\Om}-zI_\Om\big)^{-1} -
\big(H^N_{0,\Om}-zI_\Om\big)^{-1}v
\Big[I_\Om+\ol{u\big(H^N_{0,\Om}-zI_\Om\big)^{-1}v}\,\Big]^{-1}
u\big(H^N_{0,\Om}-zI_\Om\big)^{-1}, \no \\
& \hspace*{9.1cm}  z\in\bbC\backslash\sigma\big(H^N_{\Om}\big).  \lb{2.21}
\end{align}
Here invertibility of $\Big[I_\Om+\ol{u\big(H^{D,N}_{0,\Om}-zI_\Om\big)^{-1}v}\,\Big]$ for $z\in \rho\big(H^{D,N}_{\Om}\big)$ is guaranteed by arguments discussed, for instance, in \cite[Sect.\ 2]{GLMZ05} and the literature cited therein. 

Although we will not explicitly use the following result in this paper, we feel
it is of sufficient independent interest to be included at the end of this section:

\begin{lemma}\label{L-1}
Assume $\Om$ satisfies Hypothesis \ref{h2.1} with $n=2,3$ replaced by
$n\in\bbN$, $n\geq 2$, suppose that $V\in L^{n/2}(\Omega; d^nx)$, and let
\begin{equation}\label{SSS}
s\in\left\{
\begin{array}{l}
(0,1)\mbox{ if }n=2,
\\[6pt]
[0,{\textstyle{\frac{3}{2}}})\mbox{ if }n=3,
\\[6pt]
[0,2)\mbox{ if }n=4,
\\[6pt]
[0,2]\mbox{ if }n\geq 5.
\end{array}
\right.
\end{equation}
Then the operator
\begin{equation}\label{V-2}
V:H^s(\Omega)\rightarrow \big(H^{2-s}(\Omega)\big)^*
\end{equation}
is well-defined and bounded, in fact, it is compact.
\end{lemma}
\begin{proof} The fact that the operator (\ref{V-2}) is
well-defined along with the estimate
\begin{equation}\label{V-3}
\|V\|_{{\mathcal{B}}(H^s(\Omega),(H^{2-s}(\Omega))^*)} \leq
C(n,\Omega)\|V\|_{L^{n/2}(\Omega)},
\end{equation}
are direct consequences of standard embedding results (cf.\ \cite[Sect.\ 3.3.1]{Tr83} for smooth domains and \cite{Tr02} for arbitrary (bounded or unbounded) Lipschitz domains). Once the boundedness of (\ref{V-2}) has been established, the compactness follows from the fact that if $V_j\in C^\infty_0(\Omega)$ is a
sequence of functions with the property that
$V_j \underset{j\uparrow\infty}{\rightarrow} V$ in
$L^{n/2}(\Omega)$, then $V_j \underset{j\uparrow\infty}{\rightarrow} V$ in
${\mathcal{B}}\big(H^s(\Omega),\big(H^{2-s}(\Omega)\big)^*\big)$ by (\ref{V-3})
and each operator $V_j:H^s(\Omega)\to \big(H^{2-s}(\Omega)\big)^*$ is
compact, by Rellich's selection lemma (cf.\ \cite{ET89} for smooth domains and
\cite{Tr02} for arbitrary Lipschitz domains). Thus, the operator in (\ref{V-2}) is
compact as the operator norm limit of a sequence of compact operators.
\end{proof}

\section{Dirichlet and Neumann boundary value problems \\
and Dirichlet-to-Neumann maps} \label{s3}

This section is devoted to Dirichlet and Neumann boundary value problems associated with the Helmholtz differential expression $-\Delta - z$ as well as the corresponding differential expression $-\Delta + V - z$ in the presence of a potential $V$, both in connection with the open set $\Omega$. In addition, we provide a detailed discussion of Dirichlet-to-Neumann, $M^D_{0,\Om}$,
$M^D_{\Om}$, and Neumann-to-Dirichlet maps, $M^N_{0,\Om}$, $M^N_{\Om}$, in 
$L^2(\partial\Om; d^{n-1}\sigma)$.

Denote by
 \begin{equation}
 \wti\ga_N : \big\{u\in H^1(\Om) \,\big|\, \Delta u \in
\big(H^1(\Om)\big)^*\big\} \to H^{-1/2}(\dOm)  \lb{3.0}
\end{equation}
a weak Neumann trace operator defined by
\begin{align} \lb{3.1a}
\langle\wti\ga_N u, \phi\rangle = \int_\Om d^n x \, \nabla
u(x)\cdot\nabla \Phi(x)  + \langle\Delta u, \Phi\rangle
\end{align}
for all $\phi\in H^{1/2}(\dOm)$ and $\Phi\in H^1(\Om)$ such that
$\ga_D\Phi = \phi$. We note that this definition is
independent of the particular extension $\Phi$ of $\phi$, and that
$\wti\ga_N$ is an extension of the Neumann trace operator
$\ga_N$ defined in \eqref{2.3}. For more details we refer to 
\cite[App.\ A]{GLMZ05}.

We start with a basic result on the Helmholtz Dirichlet and Neumann boundary value problems:

\begin{theorem} [\cite{GMZ06}] \lb{t3.1}
Assume Hypothesis \ref{h2.1}. Then for every 
$f \in H^1(\dOm)$ and $z\in\bbC\big\backslash\si\big(H_{0,\Om}^D\big)$ the
following Dirichlet boundary value problem,
\begin{align} \lb{3.1}
\begin{cases}
(-\Delta - z)u_0^D = 0 \text{ on }\, \Om, \quad u_0^D \in H^{3/2}(\Om), \\
\ga_D u_0^D = f \text{ on }\, \dOm,
\end{cases}
\end{align}
has a unique solution $u_0^D$ satisfying $\wti\ga_N u_0^D \in
\LdOm$. Moreover, there exist constants $C^D=C^D(\Omega,z)>0$ such that
\begin{equation}
\|u_0^D\|_{H^{3/2}(\Omega)} \leq C^D \|f\|_{H^1(\partial\Omega)}.  \lb{3.3a}
\end{equation} 
Similarly, for every $g\in\LdOm$ and
$z\in\bbC\backslash\si\big(H_{0,\Om}^N\big)$ the following Neumann
boundary value problem,
\begin{align} \lb{3.2}
\begin{cases}
(-\Delta - z)u_0^N = 0 \text{ on }\,\Om,\quad u_0^N \in H^{3/2}(\Om), \\
\wti\ga_N u_0^N = g\text{ on }\,\dOm,
\end{cases}
\end{align}
has a unique solution  $u_0^N$ satisfying $\ga_D u^N_0 \in H^1(\dOm)$. Moreover, there exist constants $C^N=C^N(\Omega,z)>0$ such that
\begin{equation}
\|u_0^N\|_{H^{3/2}(\Omega)} \leq C^N \|g\|_{\LdOm}.  \lb{3.4a}
\end{equation}
In addition,  \eqref{3.1}--\eqref{3.4a} imply that the following maps are bounded  
\begin{align}
\big[\ga_N\big(\big(H^D_{0,\Om}-zI_\Om\big)^{-1}\big)^*\big]^* &: H^1(\dOm) \to
H^{3/2}(\Om), \quad z\in\bbC\big\backslash\si\big(H^D_{0,\Om}\big),   \lb{3.4b}
\\
\big[\ga_D \big(\big(H^N_{0,\Om}-zI_\Om\big)^{-1}\big)^*\big]^* &: \LdOm \to
H^{3/2}(\Om), \quad z\in\bbC\big\backslash\si\big(H^N_{0,\Om}\big).   \lb{3.4c}
\end{align} 
Finally, the solutions $u_0^D$ and $u_0^N$ are given by the formulas
\begin{align}
u_0^D (z) &= -\big(\ga_N \big(H_{0,\Om}^D-\ol{z}I_\Om\big)^{-1}\big)^*f,
\lb{3.3}
\\
u_0^N (z) &= \big(\ga_D \big(H_{0,\Om}^N-\ol{z}I_\Om\big)^{-1}\big)^*g.
\lb{3.4}
\end{align}
\end{theorem}

A detailed proof of Theorem \ref{t3.1} will appear in \cite{GMZ06}.

We temporarily strengthen our hypothesis on $V$ and introduce the following assumption:

\begin{hypothesis} \lb{h3.2}
Suppose the set $\Om$ satisfies Hypothesis \ref{h2.1} and assume that
$V\in \LOm\cap L^p(\Om;d^nx)$ for some $p>2$.
\end{hypothesis}

By employing a perturbative approach, we now extend Theorem \ref{t3.1} in connection with the Helmholtz differential expression $-\Delta - z$ on
$\Om$ to the case of a Schr\"odinger  differential expression  $-\Delta + V - z$ on $\Om$.

\begin{theorem} \lb{t3.3}
Assume Hypothesis \ref{h3.2}. Then for every $f \in
H^1(\dOm)$ and
$z\in\bbC\big\backslash\si\big(H_{\Om}^D\big)$
the following Dirichlet boundary value problem,
\begin{align} \lb{3.9}
\begin{cases}
(-\Delta + V - z)u^D = 0 \text{ on }\, \Om, \quad u^D \in H^{3/2}(\Om), \\
\ga_D u^D = f \text{ on }\, \dOm,
\end{cases}
\end{align}
has a unique solution $u^D$ satisfying $\wti\ga_N u^D \in
\LdOm$. Similarly, for every $g\in\LdOm$ and
$z\in\bbC\big\backslash\si\big(H_{\Om}^N\big)$
the following Neumann boundary value problem,
\begin{align} \lb{3.10}
\begin{cases}
(-\Delta + V - z)u^N = 0 \text{ on }\, \Om,\quad u^N \in H^{3/2}(\Om), \\
\wti\ga_N u^N = g\text{ on }\, \dOm,
\end{cases}
\end{align}
has a unique solution  $u^N$. Moreover, the solutions $u^D$ and
$u^N$ are given by the formulas
\begin{align}
u^D (z) &= -\big[\ga_N \big(\big(H_{\Om}^D-zI_\Om\big)^{-1}\big)^*\big]^*f,
\lb{3.11}
\\
u^N (z) &= \big[\ga_D \big(\big(H_{\Om}^N-zI_\Om\big)^{-1}\big)^*\big]^*g.
\lb{3.12}
\end{align}
\end{theorem}
\begin{proof}
We temporarily assume that
$z\in\bbC\big\backslash\big(\si\big(H_{0,\Om}^D\big)\cup\si\big(H_{\Om}^D\big)\big)$ in the case of the Dirichlet problem and
$z\in\bbC\big\backslash\big(\si\big(H_{0,\Om}^N\big)\cup\si\big(H_{\Om}^N\big)\big)$ in the context of the Neumann problem. \\
Uniqueness follows from the fact that $z\notin\si(H_\Om^D)$
and $z\notin\si(H_\Om^N)$, respectively.

Next, we will show that the functions
\begin{align}
u^D (z) &= u_0^D (z) - \big(H_\Om^D-zI_\Om\big)^{-1} V u_0^D (z), \lb{3.13}
\\
u^N (z)&= u_0^N (z) - \big(H_\Om^N-zI_\Om\big)^{-1} V u_0^N (z),  \lb{3.14}
\end{align}
with $u_0^D, u_0^N$ given by Theorem \ref{t3.1},
satisfy \eqref{3.11} and \eqref{3.12}, respectively. Indeed, it
follows from Theorem \ref{t3.1} that $u_0^D,u_0^N\in H^{3/2}(\Om)$
and $\wti\ga_N u_0^D \in \LdOm$. Using the Sobolev embedding theorem
$H^{3/2}(\Om) \hookrightarrow L^q(\Om;d^nx)$, $q\geq2$, and the
fact that $V\in L^p(\Om;d^nx)$, $p>2$, one concludes that $Vu_0^D,
Vu_0^N\in\LOm$, and hence \eqref{3.13} and \eqref{3.14} are 
well-defined. Since one also has $V\in\LOm$, it follows from Lemma
\ref{l2.3} that $V\big(H_{0,\Om}^D-zI_\Om\big)^{-1}$ and
$V\big(H_{0,\Om}^N-zI_\Om\big)^{-1}$ are Hilbert--Schmidt, and hence
\begin{align}
\big[I+V\big(H_{0,\Om}^D-zI_\Om\big)^{-1}\big]^{-1} \in \cB\big(\LOm\big), \quad
z\in\bbC\big\backslash\big(\si\big(H_{0,\Om}^D\big)\cup\si\big(H_{\Om}^D\big)\big),
\lb{3.15}
\\
\big[I+V\big(H_{0,\Om}^N-zI_\Om\big)^{-1}\big]^{-1} \in \cB\big(\LOm\big), \quad
z\in\bbC\big\backslash\big(\si\big(H_{0,\Om}^N\big)\cup\si\big(H_{\Om}^N\big)\big). 
\lb{3.16}
\end{align}
Thus, by \eqref{2.4} and \eqref{2.5},
\begin{align}
\big(H_\Om^D-zI_\Om\big)^{-1} V u_0^D &=
\big(H_{0,\Om}^D-zI_\Om\big)^{-1}\big[I+V\big(H_{0,\Om}^D-zI_\Om\big)^{-1}\big]^{-1}Vu_0^D
\in H^2(\Om),
\\
\big(H_\Om^N-zI_\Om\big)^{-1} V u_0^N &=
\big(H_{0,\Om}^N-zI_\Om\big)^{-1}\big[I+V\big(H_{0,\Om}^N-zI_\Om\big)^{-1}\big]^{-1}Vu_0^N
\in H^2(\Om),
\end{align}
and hence $u^D,u^N\in H^{3/2}(\Om)$ and $\wti\ga_N u^D \in \LdOm$.
Moreover,
\begin{align}
(-\Delta+V-z)u^D &= (-\Delta-z)u_0^D + Vu_0^D -
(-\Delta+V-z)\big(H_{\Om}^D-zI_\Om\big)^{-1}Vu_0^D \no
\\ &=
Vu_0^D - I_\Om Vu_0^D = 0,
\\
(-\Delta+V-z)u^N &= (-\Delta-z)u_0^N + Vu_0^N -
(-\Delta+V-z)\big(H_{\Om}^N-zI_\Om\big)^{-1}Vu_0^N \no
\\ &=
Vu_0^N - I_\Om Vu_0^N = 0,
\end{align}
and by \eqref{2.4}, \eqref{2.5} and \eqref{3.15}, \eqref{3.16} one
also obtains,
\begin{align}
\ga_D u^D &= \ga_D u_0^D - \ga_D\big(H_{\Om}^D-zI_\Om\big)^{-1}Vu_0^D \no
\\&=
f - \ga_D \big(H_{0,\Om}^D-zI_\Om\big)^{-1}
\big[I+V\big(H_{0,\Om}^D-zI_\Om\big)^{-1}\big]^{-1} Vu_0^D=f,
\\
\wti\ga_N u^N &= \wti\ga_N u_0^N -
\wti\ga_N\big(H_{\Om}^N-zI_\Om\big)^{-1}Vu_0^N \no
\\&=
g - \ga_N \big(H_{0,\Om}^N-zI_\Om\big)^{-1}
\big[I+V\big(H_{0,\Om}^N-zI_\Om\big)^{-1}\big]^{-1} Vu_0^N=g.
\end{align}

Finally, \eqref{3.11} and \eqref{3.12} follow from \eqref{3.3},
\eqref{3.4}, \eqref{3.13}, \eqref{3.14}, and the resolvent
identity,
\begin{align}
u^D (z) &= \big[I_\Om - \big(H_\Om^D-zI_\Om\big)^{-1} V\big] \big[-\ga_N
\big(\big(H_{0,\Om}^D-zI_\Om\big)^{-1}\big)^*\big]^*f \no
\\ &=
-\big[\ga_N \big(\big(H_{0,\Om}^D- zI_\Om\big)^{-1}\big)^*
\big[I_\Om - \big(H_\Om^D-zI_\Om\big)^{-1} V\big]^*\big]^*f \no
\\ &=
-\big[\ga_N \big(\big(H_\Om^D-zI_\Om\big)^{-1}\big)^*\big]^*f,
\\[1mm]
u^N (z) &= \big[I_\Om - \big(H_\Om^N-zI_\Om\big)^{-1} V\big] \big[\ga_D
\big(\big(H_{0,\Om}^N-zI_\Om\big)^{-1}\big)^*\big]^*g \no
\\ &=
\big[\ga_D \big(\big(H_{0,\Om}^N-zI_\Om\big)^{-1}\big)^*
\big[I_\Om - \big(H_\Om^N-zI_\Om\big)^{-1} V\big]^*\big]^*g \no
\\ &=
\big[\ga_D \big(\big(H_\Om^N-zI_\Om\big)^{-1}\big)^*\big]^*g.
\end{align}
Analytic continuation with respect to $z$ then permits one to remove the additional condition $z \notin \si\big(H_{0,\Om}^D\big)$ in the case of the Dirichlet problem, and
the additional condition $z \notin \si\big(H_{0,\Om}^N\big)$ in the context of the Neumann problem.
\end{proof}

Assuming Hypothesis \ref{h3.2}, we now introduce the
Dirichlet-to-Neumann maps, $M_{0,\Om}^{D}(z)$ and $M_\Om^{D}(z)$,
associated with $(-\Delta-z)$ and $(-\Delta+V-z)$ on $\Om$, as follows,
\begin{align}
M_{0,\Om}^{D}(z) \colon
\begin{cases}
H^1(\dOm) \to \LdOm,
\\
\hspace*{10mm} f \mapsto -\wti\ga_N u_0^D,
\end{cases}  \quad z\in\bbC\big\backslash\si\big(H_{0,\Om}^D\big),
\end{align}
where $u_0^D$ is the unique solution of
\begin{align}
(-\Delta-z)u_0^D = 0 \,\text{ on }\Om, \quad u_0^D\in
H^{3/2}(\Om), \quad \ga_D u_0^D = f \,\text{ on }\dOm,
\end{align}
and
\begin{align}
M_\Om^{D}(z) \colon
\begin{cases}
H^1(\dOm) \to \LdOm,
\\
\hspace*{10mm} f \mapsto -\wti\ga_N u^D,
\end{cases}  \quad z\in\bbC\big\backslash\si\big(H_{\Om}^D\big),
\end{align}
where $u^D$ is the unique solution of
\begin{align}
(-\Delta+V-z)u^D = 0 \,\text{ on }\Om,  \quad u^D \in
H^{3/2}(\Om), \quad \ga_D u^D= f \,\text{ on }\dOm.
\end{align}

In addition, still assuming Hypothesis \ref{h3.2}, we introduce the Neumann-to-Dirichlet maps, $M_{0,\Om}^{N}(z)$ and $M_\Om^{N}(z)$,
associated with $(-\Delta-z)$ and $(-\Delta+V-z)$ on $\Om$, as
follows,
\begin{align}
M_{0,\Om}^{N}(z) \colon \begin{cases} \LdOm \to H^1(\dOm),
\\
\hspace*{20.5mm} g \mapsto \ga_D u_0^N, \end{cases}  \quad
z\in\bbC\big\backslash\si\big(H_{0,\Om}^N\big),
\end{align}
where $u_0^N$ is the unique solution of
\begin{align}
(-\Delta-z)u_0^N = 0 \,\text{ on }\Om, \quad u_0^N\in
H^{3/2}(\Om), \quad \wti\ga_N u_0^N = g \,\text{ on }\dOm,
\end{align}
and
\begin{align}
M_\Om^{N}(z) \colon \begin{cases}
\LdOm \to H^1(\dOm),
\\
\hspace*{20.5mm} g \mapsto \ga_D u^N,
\end{cases}  \quad
z\in\bbC\big\backslash\si\big(H_{\Om}^N\big),
\end{align}
where $u^N$ is the unique solution of
\begin{align}
(-\Delta+V-z)u^N = 0 \,\text{ on }\Om,  \quad u^N \in
H^{3/2}(\Om), \quad \wti\ga_N u^N= g \,\text{ on }\dOm.
\end{align}

It follows from Theorems \ref{t3.1} and \ref{t3.3}, that under the
assumption of Hypothesis \ref{h3.2}, the operators
$M_{0,\Om}^D(z)$, $M_{\Om}^D(z)$, $M_{0,\Om}^N(z)$, and
$M_{\Om}^N(z)$ are well-defined and satisfy the following
equalities,
\begin{align}
M_{0,\Om}^{N}(z) &= - M_{0,\Om}^{D}(z)^{-1}, \quad
z\in\bbC\big\backslash\big(\si\big(H_{0,\Om}^D\big)\cup\si\big(H_{0,\Om}^N\big)\big),
\lb{3.28}
\\
M_{\Om}^{N}(z) &= - M_{\Om}^{D}(z)^{-1}, \quad
z\in\bbC\big\backslash\big(\si\big(H_{\Om}^D\big)\cup\si\big(H_{\Om}^N\big)\big),
\lb{3.29}
\intertext{and}%
M^{D}_{0,\Om}(z) &= \wti\gamma_N\big[\gamma_N
\big(\big(H^D_{0,\Om} - zI_\Om\big)^{-1}\big)^*\big]^*, \quad
z\in\bbC\big\backslash\si\big(H_{0,\Om}^D\big), \lb{3.30}
\\
M^{D}_{\Om}(z) &= \wti\gamma_N\big[\gamma_N \big(\big(H^D_{\Om} -
zI_\Om\big)^{-1}\big)^*\big]^*, \quad
z\in\bbC\big\backslash\si\big(H_{\Om}^D\big),
\lb{3.31}
\\
M^{N}_{0,\Om}(z) &= \gamma_D\big[\gamma_D
\big(\big(H^N_{0,\Om} - zI_\Om\big)^{-1}
\big)^*\big]^*, \quad
z\in\bbC\big\backslash\si\big(H_{0,\Om}^N\big), \lb{3.32}
\\
M^{N}_{\Om}(z) &= \gamma_D\big[\gamma_D
\big(\big(H^N_{\Om} - zI_\Om\big)^{-1}\big)^*\big]^*, \quad
z\in\bbC\big\backslash\si\big(H_{\Om}^N\big).
\lb{3.33}
\end{align}

The representations \eqref{3.30}--\eqref{3.33} provide a convenient point of departure for proving the operator-valued Herglotz property of
$M^{D}_{\Om}$ and $M^{N}_{\Om}$. We will return to this topic in a future paper.

Next, we note that the above formulas \eqref{3.30}--\eqref{3.33} may
be used as alternative definitions of the
Dirichlet-to-Neumann and Neumann-to-Dirichlet maps. In
particular, we will next use \eqref{3.31} and \eqref{3.33} to extend
the above definition of the operators $M^{D}_{\Om}(z)$ and
$M^{N}_{\Om}(z)$ to the more general situation governed by Hypothesis \ref{h2.6}:

\begin{lemma} [\cite{GMZ06}] \lb{l3.4}
Assume Hypothesis \ref{h2.6}. Then the operators
$M^{D}_{\Om}(z)$ and $M^{N}_{\Om}(z)$ defined by equalities
\eqref{3.31} and \eqref{3.33} have the following boundedness 
properties,
\begin{align}
& M^{D}_{\Om}(z) \in \cB\big(H^{1}(\partial\Om),\LdOm\big), \quad
z\in\bbC\big\backslash\si\big(H_{\Om}^D\big),
\lb{3.42a} \\
& M^{N}_{\Om}(z) \in \cB\big(\LdOm,H^{1}(\partial\Om)\big), \quad
z\in\bbC\big\backslash\si\big(H_{\Om}^N\big).
\lb{3.43a}
\end{align}
\end{lemma}

A detailed proof of Lemma \ref{l3.4} will be provided in \cite{GMZ06}.

Weyl--Titchmarsh operators, in a spirit close to ours, have recently been discussed by Amrein and Pearson \cite{AP04} in connection with the interior and exterior of a ball in $\bbR^3$ and real-valued potentials $V\in L^\infty(\bbR^3;d^3x)$. For additional literature on Weyl--Titchmarsh operators, relevant in the context of boundary value spaces (boundary triples, etc.), we refer, for instance, to \cite{ABMN05}, \cite{BL06}, 
\cite{BMN06}, 
\cite{BMN00}, \cite{BMN02}, \cite{BM04}, \cite{DM91}, \cite{DM95}, \cite{GKMT01}, 
\cite[Ch.\ 3]{GG91}, \cite{MM06}, \cite{Ma04}, \cite{MPP07}, \cite{Pa87}, \cite{Pa02}. 

Next, we prove the following auxiliary result, which will play a crucial role in Theorem \ref{t4.2}, the principal result of this paper.

\begin{lemma} \lb{l3.5}
Assume Hypothesis \ref{h2.6}. Then the following identities hold,
\begin{align}
M_{0,\Om}^D(z) - M_\Om^D(z) &= \ol{\wti\gamma_N
\big(H^D_{\Om}-zI_\Om\big)^{-1} V \big[\gamma_N
\big(\big(H^D_{0,\Om}-zI_\Om\big)^{-1}\big)^*\big]^*}, \no
\\
&\hspace*{3.1cm}
z\in\bbC\big\backslash\big(\si\big(H_{0,\Om}^D\big)\cup\si\big(H_{\Om}^D\big)\big),
\lb{3.35}
\\
M_\Om^D(z) M_{0,\Om}^D(z)^{-1} &= I_\dOm - \ol{\wti\gamma_N
\big(H^D_{\Om}-zI_\Om\big)^{-1} V \big[\gamma_D
\big(\big(H^N_{0,\Om}-zI_\Om\big)^{-1}\big)^*\big]^*}, \no
\\
&\hspace*{2.45cm}
z\in\bbC\big\backslash\big(\si\big(H_{0,\Om}^D\big)\cup\si\big(H_{\Om}^D\big)
\cup\si\big(H_{0,\Om}^N\big)\big). \lb{3.36}
\end{align}
\end{lemma}
\begin{proof}
Let $z\in\bbC\big\backslash\big(\si\big(H_{0,\Om}^D\big)\cup\si\big(H_{\Om}^D\big)\big)$.
Then \eqref{3.35} follows from \eqref{3.30}, \eqref{3.31}, and
the resolvent identity
\begin{align}
M_{0,\Om}^D(z) - M_\Om^D(z) &=
\wti\gamma_N\big[\gamma_N\big(\big(H^D_{0,\Om}-zI_\Om\big)^{-1} -
\big(H^D_{\Om}-zI_\Om\big)^{-1}\big)^*\big]^* \no
\\ &=
\ol{\wti\gamma_N\big[\gamma_N \big(
\big(H^D_{\Om}-zI_\Om\big)^{-1}V\big(H^D_{0,\Om}-zI_\Om\big)^{-1} \big)^*\big]^*}
\\ &=
\ol{\wti\gamma_N \big(H^D_{\Om}-zI_\Om\big)^{-1} V \big[\gamma_N
\big(\big(H^D_{0,\Om}-zI_\Om\big)^{-1}\big)^*\big]^*}. \no
\end{align}
Next, if 
$z\in\bbC\big\backslash\big(\si\big(H_{0,\Om}^D\big)\cup\si\big(H_{\Om}^D\big)
\cup\si\big(H_{0,\Om}^N\big)\big)$, then it follows from \eqref{3.28},
\eqref{3.32}, and \eqref{3.35} that
\begin{align} \lb{3.40}
M_\Om^D(z) M_{0,\Om}^D(z)^{-1} &= I_\dOm + \big(M_\Om^D(z) -
M_{0,\Om}^D(z)\big)M_{0,\Om}^D(z)^{-1} \no
\\ &=
I_\dOm + \big(M_{0,\Om}^D(z) - M_{\Om}^D(z)\big)M_{0,\Om}^N(z)
\no
\\ &=
I_\dOm + \ol{\wti\gamma_N \big(H^D_{\Om}-zI_\Om\big)^{-1} V \big[\gamma_N
\big(\big(H^D_{0,\Om}-zI_\Om\big)^{-1}\big)^*\big]^*}
\\ &\quad \times
\gamma_D\big[\gamma_D
\big(\big(H^N_{0,\Om}-zI_\Om\big)^{-1}\big)^*\big]^*. \no
\end{align}
Let $g\in\LdOm$. Then by Theorem \ref{t3.1},
\begin{align}
u=\big[\gamma_D\big(\big(H^N_{0,\Om}-zI_\Om\big)^{-1}\big)^*\big]^*g
\lb{3.41}
\end{align}
is the unique solution of
\begin{align}
(-\Delta-z)u = 0 \,\text{ on }\Om, \quad u\in H^{3/2}(\Om),
\quad \wti\ga_N u = g \,\text{ on }\dOm.
\end{align}
Setting $f=\ga_D u \in H^1(\dOm)$ and utilizing Theorem
\ref{t3.1} once again, one obtains 
\begin{align}
u &= -\big[\ga_N \big(H_{0,\Om}^D-\ol{z}I_\Om\big)^{-1}\big]^*f \no
\\ &=
-\big[\gamma_N \big(\big(H^D_{0,\Om}-zI_\Om\big)^{-1}\big)^*\big]^*
\gamma_D\big[\gamma_D
\big(\big(H^N_{0,\Om}-zI_\Om\big)^{-1}\big)^*\big]^*g. \lb{3.43}
\end{align}
Thus, it follows from \eqref{3.41} and \eqref{3.43} that
\begin{align}
\big[\gamma_N \big(\big(H^D_{0,\Om}-zI_\Om\big)^{-1}\big)^*\big]^*
\gamma_D\big[\gamma_D
\big(\big(H^N_{0,\Om}-zI_\Om\big)^{-1}\big)^*\big]^* =
-\big[\gamma_D\big(\big(H^N_{0,\Om}-zI_\Om\big)^{-1}\big)^*\big]^*.
\lb{3.44}
\end{align}
Finally, insertion of \eqref{3.44} into \eqref{3.40} yields
\eqref{3.36}.
\end{proof}

It follows from \eqref{4.24}--\eqref{4.29} that $\widetilde \gamma_N$ can be replaced by $\gamma_N$ on the right-hand side of \eqref{3.35} and \eqref{3.36}.

\section{A Multi-Dimensional Variant of a Formula due to Jost and Pais}
\label{s4}

In this section we prove our multi-dimensional variants of the Jost and Pais formula as discussed in the introduction.

We start with an elementary comment on determinants which, however, lies at the heart of the matter of our multi-dimensional variant of the one-dimensional Jost and Pais result. Suppose $A \in \cB(\cH_1, \cH_2)$, $B \in \cB(\cH_2, \cH_1)$ with $A B \in \cB_1(\cH_2)$ and $B A \in \cB_1(\cH_1)$. Then,
\begin{equation}
\det (I_{\cH_2}-AB) = \det (I_{\cH_1}-BA).   \lb{4.0}
\end{equation}
In particular, $\cH_1$ and $\cH_2$ may have different dimensions. Especially, one of them may be infinite and the other finite, in which case one of the two determinants in \eqref{4.0} reduces to a finite determinant. This case indeed occurs in the original one-dimensional case studied by Jost and Pais \cite{JP51} as described in detail in 
\cite{GM03} and the references therein. In the proof of the next theorem, the role of 
$\cH_1$ and $\cH_2$ will be played by $L^2(\Om; d^n x)$ and 
$L^2(\partial\Om; d^{n-1} \sigma)$, respectively.

We start with an extension of a result in \cite{GLMZ05}:

\begin{theorem} \lb{t4.1}
Assume Hypothesis \ref{h2.6} and let
$z\in\bbC\big\backslash\big(\si\big(H_{\Om}^D\big)\cup \si\big(H_{0,\Om}^D\big) \cup
\si\big(H_{0,\Om}^N\big)\big)$. Then,
\begin{align}
\ol{\ga_N\big(H_{0,\Om}^D-zI_{\Om}\big)^{-1}V \big(H_{\Om}^D-zI_{\Om}\big)^{-1}V
\big[\ga_D \big(H_{0,\Om}^N-\ol{z}I_{\Om}\big)^{-1}\big]^*}
&\in\cB_1\big(\LdOm\big),  \lb{4.1}  \\
\ol{\ga_N\big(H_{\Om}^D-zI_{\Om}\big)^{-1}V
\big[\ga_D \big(H_{0,\Om}^N-\ol{z}I_{\Om}\big)^{-1}\big]^*}
&\in\cB_2\big(\LdOm\big), \lb{4.2}
\end{align}
and
\begin{align}
&
\frac{\det{}_2\Big(I_{\Om}+\ol{u\big(H_{0,\Om}^N-zI_{\Om}\big)^{-1}v}\,\Big)}
{\det{}_2\Big(I_{\Om}+\ol{u\big(H_{0,\Om}^D-zI_{\Om}\big)^{-1}v}\,\Big)}  \no \\
&\quad = \det{}_2\Big(I_{\dOm} -
\ol{\ga_N\big(H_{\Om}^D-zI_{\Om}\big)^{-1}V
\big[\ga_D \big(H_{0,\Om}^N-\ol{z}I_{\Om}\big)^{-1}\big]^*} \, \Big)
\lb{4.3}\\
&\quad\quad \times \exp\Big(\tr\Big(\, 
\ol{\ga_N\big(H_{0,\Om}^D-zI_{\Om}\big)^{-1}V \big(H_{\Om}^D-zI_{\Om}\big)^{-1}V
\big[\ga_D \big(H_{0,\Om}^N-\ol{z}I_{\Om}\big)^{-1}\big]^*} \,\Big)\Big). \no
\end{align}
\end{theorem}
\begin{proof}
From the outset we note that the left-hand side of \eqref{4.3} is
well-defined by \eqref{2.35}. Let
$z\in\bbC\big\backslash\big(\si\big(H_{\Om}^D\big)\cup\si\big(H_{0,\Om}^D\big) \cup
\si\big(H_{0,\Om}^N\big)\big)$ and
\begin{align}
u(x) &= \exp(i\arg(V(x)))\abs{V(x)}^{1/2},\quad
v(x)=\abs{V(x)}^{1/2},
\\
\wti u(x) &= \exp(i\arg(V(x)))\abs{V(x)}^{p/p_1},\quad \wti
v(x)=\abs{V(x)}^{p/p_2},
\end{align}
where
\begin{align} \lb{4.6}
p_1=\begin{cases} 3p/2,&n=2, \\  4p/3,&n=3,\end{cases} \qquad
p_2=\begin{cases}3p,&n=2, \\ 4p,& n=3  \end{cases}
\end{align} 
with $p$ as introduced in Hypothesis \ref{h2.6}. 
Then it follows that $\f1{p_1}+\f{1}{p_2}=\f1p$, in both cases
$n=2,3$, and hence $V=uv=\wti u \wti v$.

Next, we introduce
\begin{equation}
K_D(z)=-\ol{u\big(H_{0,\Om}^D-zI_{\Om}\big)^{-1}v}, \quad
K_N(z)=-\ol{u\big(H_{0,\Om}^N-zI_{\Om}\big)^{-1}v}
\end{equation}
and note that
\begin{align}
[I_{\Om}-K_D(z)]^{-1} \in\cB\big(\LOm\big), \quad
z\in\bbC\big\backslash\big(\si\big(H_{\Om}^D\big)\cup\si\big(H_{0,\Om}^D\big)\big). 
\end{align}
Thus, utilizing the following facts,
\begin{align}
[I_{\Om} - K_D(z)]^{-1} &= I_{\Om} + K_D(z)[I_{\Om} - K_D(z)]^{-1}
\end{align}
and
\begin{align}
1& = \det{}_2\big([I_{\Om}-K_D(z)][I_{\Om} - K_D(z)]^{-1}\big)
\\
&=
\det{}_2\big(I_{\Om}-K_D(z)\big)\det{}_2\big([I_{\Om} - K_D(z)]^{-1}\big)
\exp\big(\tr\big(K_D(z)^2[I_{\Om} - K_D(z)]^{-1}\big)\big),  \no
\end{align}
one obtains
\begin{align}
&\det{}_2\big([I_{\Om}-K_N(z)][I_{\Om}-K_D(z)]^{-1}\big)  \no \\
&\quad = \det{}_2\big(I_{\Om}-K_N(z)\big)
\det{}_2\big([I_{\Om}-K_D(z)]^{-1}\big)
\exp\big(\tr\big(K_N(z)K_D(z)[I_{\Om}-K_D(z)]^{-1}\big)\big) \\
&\quad =
\frac{\det{}_2\big(I_{\Om}-K_N(z)\big)}{\det{}_2\big(I_{\Om}-K_D(z)\big)}
\exp\big(\tr\big((K_N(z)-K_D(z))K_D(z)[I_{\Om}-K_D(z)]^{-1}\big)\big).
\no
\end{align}
At this point, the left-hand side of \eqref{4.3} can be rewritten
as
\begin{align}
&\frac{\det{}_2\Big(I_{\Om}+\ol{u\big(H_{0,\Om}^N-zI_{\Om}\big)^{-1}v}\,\Big)}
{\det{}_2\Big(I_{\Om}+\ol{u\big(H_{0,\Om}^D-zI_{\Om}\big)^{-1}v}\,\Big)} =
\frac{\det{}_2\big(I_{\Om}-K_N(z)\big)}{\det{}_2\big(I_{\Om}-K_D(z)\big)}
\no \\
&\quad = \det{}_2\big([I_{\Om}-K_N(z)][I_{\Om}-K_D(z)]^{-1}\big)
\no \\
&\quad\quad \, \times
\exp\big(\tr\big((K_D(z)-K_N(z))K_D(z)[I_{\Om}-K_D(z)]^{-1}\big)\big)
\no  \\
&\quad =
\det{}_2\big(I_{\Om}+(K_D(z)-K_N(z))[I_{\Om}-K_D(z)]^{-1}\big)
\lb{4.12}
\\
&\quad\quad \, \times
\exp\big(\tr\big((K_D(z)-K_N(z))K_D(z)[I_{\Om}-K_D(z)]^{-1}\big)\big).
\no
\end{align}
Next, temporarily suppose that $V\in L^p(\Om;d^nx)\cap
L^\infty(\Om;d^nx)$. Using \cite[Lemma A.3]{GLMZ05} (an extension of a
result of Nakamura \cite[Lemma 6]{Na01}) and \cite[Remark A.5]{GLMZ05}, 
one finds
\begin{align}
\begin{split}
K_D(z)-K_N(z) &=
-\ol{u\big[\big(H_{0,\Om}^D-zI_{\Om}\big)^{-1}- \big(H_{0,\Om}^N
-zI_{\Om}\big)^{-1}\big]v}
\\ &=
-\ol{u\big[\ga_D \big(H_{0,\Om}^N-\ol{z}I_{\Om}\big)^{-1}\big]^*}\,
\ol{\ga_N \big(H_{0,\Om}^D -zI_{\Om}\big)^{-1}v}
\\ &=
-\Big[\,\ol{\ga_D \big(H_{0,\Om}^N-\ol{z}I_{\Om}\big)^{-1}\ol{u}}\,\Big]^* \,
\ol{\ga_N \big(H_{0,\Om}^D -zI_{\Om}\big)^{-1}v}.
\end{split}\lb{4.13}
\end{align}
Thus, inserting \eqref{4.13} into \eqref{4.12} yields,
\begin{align} \lb{4.14}
&\frac{\det{}_2\Big(I_{\Om}+\ol{u\big(H_{0,\Om}^N-zI_{\Om}\big)^{-1}v}\,\Big)}
{\det{}_2\Big(I_{\Om}+\ol{u\big(H_{0,\Om}^D-zI_{\Om}\big)^{-1}v}\,\Big)} \no
\\
&\quad = \det{}_2\Big(I_{\Om} - \Big[\,\ol{\ga_D
\big(H_{0,\Om}^N-\ol{z}I_{\Om}\big)^{-1}\ol{u}}\,\Big]^* \ol{\ga_N
\big(H_{0,\Om}^D-zI_{\Om}\big)^{-1}v}  \no
\Big[I_{\Om}+\ol{u\big(H_{0,\Om}^D-zI_{\Om}\big)^{-1}v}\,\Big]^{-1}\Big)
\no
\\
&\quad\quad \times \exp\Big(\tr\Big(\Big[\,\ol{\ga_D
\big(H_{0,\Om}^N-\ol{z}I_{\Om}\big)^{-1}\ol{u}}\,\Big]^*
\ol{\ga_N\big(H_{0,\Om}^D-zI_{\Om}\big)^{-1}v}  \no 
\\
& \hspace*{2.6cm} \times
\ol{u\big(H_{0,\Om}^D-zI_{\Om}\big)^{-1}v}\Big[I_{\Om}+\ol{u\big(H_{0,\Om}^D
-zI_{\Om}\big)^{-1}v}\,\Big]^{-1}\Big)\Big). 
\end{align}
Then, utilizing Corollary \ref{c2.5} with $p_1$ and $p_2$ as in
\eqref{4.6}, one finds,
\begin{align}
\ol{\ga_D \big(H_{0,\Om}^N-\ol{z}I_{\Om}\big)^{-1}\ol{u}}
&\in\cB_{p_1}\big(\LOm,\LdOm\big),
\\
\ol{\ga_N\big(H_{0,\Om}^D-zI_{\Om}\big)^{-1}v}
&\in\cB_{p_2}\big(\LOm,\LdOm\big),
\end{align}
and hence,
\begin{align}
\Big[\,\ol{\ga_D \big(H_{0,\Om}^N-\ol{z}I_{\Om}\big)^{-1}\ol{u}}\,\Big]^*
\ol{\ga_N \big(H_{0,\Om}^D-zI_{\Om}\big)^{-1}v} &\in \cB_p\big(\LOm\big)
\subset\cB_2\big(\LOm\big),
\\
\ol{\ga_N \big(H_{0,\Om}^D-zI_{\Om}\big)^{-1}v} \Big[\,\ol{\ga_D
\big(H_{0,\Om}^N-\ol{z}I_{\Om}\big)^{-1}\ol{u}}\,\Big]^* &\in
\cB_p\big(\LdOm\big) \subset\cB_2\big(\LdOm\big).
\end{align}
Moreover, using the fact that
\begin{align}
\Big[I_{\Om}+\ol{u\big(H_{0,\Om}^D-zI_{\Om}\big)^{-1}v}\,\Big]^{-1} \in
\cB\big(\LOm\big), \quad z\in\bbC\big\backslash
\big(\si\big(H_{\Om}^D\big)\cup\si\big(H_{0,\Om}^D\big)\big),
\end{align}
one now applies the idea expressed in formula \eqref{4.0} and rearranges the terms in \eqref{4.14} as follows:
\begin{align} \lb{4.20}
&\frac{\det{}_2\Big(I_{\Om}+\ol{u\big(H_{0,\Om}^N-zI_{\Om}\big)^{-1}v}\,\Big)}
{\det{}_2\Big(I_{\Om}+\ol{u\big(H_{0,\Om}^D-zI_{\Om}\big)^{-1}v}\,\Big)} \no
\\
&\quad = \det{}_2\Big(I_{\dOm} - \ol{\ga_N
\big(H_{0,\Om}^D-zI_{\Om}\big)^{-1}v}
\Big[I_{\Om}+\ol{u\big(H_{0,\Om}^D-zI_{\Om}\big)^{-1}v}\,\Big]^{-1}
\Big[\, \ol{ \ga_D
\big(H_{0,\Om}^N-\ol{z}I_{\Om}\big)^{-1}\ol{u}}\,\Big]^*\Big)  \no
\\
&\quad\quad \times \exp\Big(\tr\Big(\,\ol{\ga_N
\big(H_{0,\Om}^D-zI_{\Om}\big)^{-1} v} \; \ol{u\big(H_{0,\Om}^D-zI_{\Om}\big)^{-1}
v}  \no
\\
& \hspace*{2.45cm} \times \Big[I_{\Om}+
\ol{u\big(H_{0,\Om}^D-zI_{\Om}\big)^{-1} v}\,\Big]^{-1} \Big[\,\ol{\ga_D
\big(H_{0,\Om}^N-\ol{z}I_{\Om}\big)^{-1}\ol{u}}\,\Big]^*\Big)\Big) \no
\\
&\quad = \det{}_2\Big(I_{\dOm} - \ol{\ga_N \big(H_{0,\Om}^D
-zI_{\Om}\big)^{-1}\wti v} \Big[I_{\Om}+\ol{\wti
u\big(H_{0,\Om}^D-zI_{\Om}\big)^{-1}\wti v}\,\Big]^{-1} \Big[\,\ol{\ga_D
\big(H_{0,\Om}^N-\ol{z}I_{\Om}\big)^{-1}\ol{\wti u}}\,\Big]^*\Big) \no
\\
&\quad\quad \times
\exp\Big(\tr\Big(\,\ol{\ga_N\big(H_{0,\Om}^D-zI_{\Om}\big)^{-1}\wti v} \;
\ol{\wti u\big(H_{0,\Om}^D-zI_{\Om}\big)^{-1}\wti v}
\\
& \hspace*{2.45cm} \times \Big[I_{\Om}+\ol{\wti
u\big(H_{0,\Om}^D-zI_{\Om}\big)^{-1}\wti v}\,\Big]^{-1} \Big[\,\ol{\ga_D
\big(H_{0,\Om}^N-\ol{z}I_{\Om}\big)^{-1}\ol{\wti u}}\,\Big]^*\Big)\Big).
\no
\end{align}
In the last equality we employed the following simple
identities,
\begin{align}
& V = u v = \wti u  \wti v,  \\
& v\Big[I_{\Om}+\ol{u\big(H_{0,\Om}^D-zI_{\Om}\big)^{-1}v}\,\Big]^{-1}u =
\wti v \Big[I+\ol{\wti u\big(H_{0,\Om}^D-zI_{\Om}\big)^{-1}\wti
v}\,\Big]^{-1} \wti u.
\end{align}
Utilizing \eqref{4.20} and the following resolvent identity,
\begin{align} \lb{4.23}
\ol{\big(H_{\Om}^D-zI_{\Om}\big)^{-1}\wti v} =
\ol{\big(H_{0,\Om}^D-zI_{\Om}\big)^{-1}\wti v} \Big[I_{\Om}+\ol{\wti
u\big(H_{0,\Om}^D-zI_{\Om}\big)^{-1}\wti v}\,\Big]^{-1},
\end{align}
one arrives at \eqref{4.3}, subject to the extra assumption $V\in
L^p(\Om;d^n x)\cap L^\infty(\Om;d^n x)$.

Finally, assuming only $V\in L^p(\Om;d^n x)$ and utilizing \cite[Thm.\ 3.2]{GLMZ05}, Lemma \ref{l2.3}, and Corollary \ref{c2.5} once again, one obtains
\begin{align}
\Big[I_{\Om}+\ol{\wti u\big(H_{0,\Om}^D-zI_{\Om}\big)^{-1} \wti
v}\,\Big]^{-1} &\in \cB\big(\LOm\big), \lb{4.24}
\\
\wti u\big(H_{0,\Om}^D-zI_{\Om}\big)^{-p/p_1} &\in
\cB_{p_1}\big(\LOm\big), \lb{4.25}
\\
\wti v\big(H_{0,\Om}^D-zI_{\Om}\big)^{-p/p_2} &\in
\cB_{p_2}\big(\LOm\big), \lb{4.26}
\\
\ol{\ga_D \big(H_{0,\Om}^N-zI_{\Om}\big)^{-1}\wti u} &\in
\cB_{p_1}\big(\LOm,\LdOm\big), \lb{4.27}
\\
\ol{\ga_N\big(H_{0,\Om}^D-zI_{\Om}\big)^{-1}\wti v} &\in
\cB_{p_2}\big(\LOm,\LdOm\big), \lb{4.28}
\end{align}
and hence
\begin{equation}
\ol{\wti u\big(H_{0,\Om}^D-zI_{\Om}\big)^{-1}\wti v} \in
\cB_p\big(\LOm\big)\subset\cB_2\big(\LOm\big).  \lb{4.29}
\end{equation}
Relations \eqref{4.23}--\eqref{4.29} prove \eqref{4.1} and
\eqref{4.2}, and hence, the left- and the right-hand sides of
\eqref{4.3} are well-defined for $V\in L^p(\Om;d^nx)$. Thus, using
\eqref{2.8}, \eqref{2.27}, \eqref{2.28}, the continuity of
$\det{}_2(\cdot)$ with respect to the Hilbert--Schmidt norm
$\|\cdot\|_{\cB_2\big(\LOm\big)}$, the continuity of $\tr(\cdot)$ with
respect to the trace norm $\|\cdot\|_{\cB_1\big(\LOm\big)}$, and an
approximation of $V\in L^p(\Om;d^nx)$ by a sequence of potentials
$V_k \in L^p(\Om;d^nx)\cap L^\infty(\Om;d^nx)$, $k\in\bbN$, in the
norm of $L^p(\Om;d^nx)$ as $k\uparrow\infty$, then extends the
result from $V\in L^p(\Om;d^nx)\cap L^\infty(\Om;d^nx)$ to $V\in
L^p(\Om;d^nx)$, $n=2,3$.
\end{proof}

Given these preparations, we are now ready for the principal result of this paper, the multi-dimensional analog of Theorem \ref{t1.2}:

\begin{theorem} \lb{t4.2}
Assume Hypothesis \ref{h2.6} and let
$z\in\bbC\big\backslash\big(\si\big(H_{\Om}^D\big)\cup \si\big(H_{0,\Om}^D\big) \cup
\si\big(H_{0,\Om}^N\big)\big)$. Then,
\begin{equation}
M_{\Om}^{D}(z)M_{0,\Om}^{D}(z)^{-1} - I_{\partial\Om} =
- \ol{\ga_N\big(H_{\Om}^D-zI_{\Om}\big)^{-1} V
\big[\ga_D(H_{0,\Om}^N-\ol{z}I_{\Om})^{-1}\big]^*}
\in \cB_2\big(L^2(\partial\Om; d^{n-1}\sigma)\big)
\end{equation}
and
\begin{align}
& \frac{\det{}_2\Big(I_{\Om}+\ol{u\big(H_{0,\Om}^N-zI_{\Om}\big)^{-1}v}\,\Big)}
{\det{}_2\Big(I_{\Om}+\ol{u\big(H_{0,\Om}^D-zI_{\Om}\big)^{-1}v}\,\Big)} \no \\
& \quad = \det{}_2\Big(I_{\dOm} - \ol{\ga_N\big(H_{\Om}^D-zI_{\Om}\big)^{-1} V
\big[\ga_D(H_{0,\Om}^N-\ol{z}I_{\Om})^{-1}\big]^*}\,\Big)
e^{\tr(T_2(z))}   \lb{4.30}  \\
& \quad =
\det{}_2\big(M_{\Om}^{D}(z)M_{0,\Om}^{D}(z)^{-1}\big)
e^{\tr(T_2(z))},   \lb{4.31}
\end{align}
where
\begin{equation}
T_2(z)=\ol{\ga_N\big(H_{0,\Om}^D-zI_{\Om}\big)^{-1} V
\big(H_{\Om}^D-zI_{\Om}\big)^{-1} V
\big[\ga_D \big(H_{0,\Om}^N-\ol{z}I_{\Om}\big)^{-1}\big]^*}
\in \cB_1\big(L^2(\partial\Om; d^{n-1}\sigma)\big).
\end{equation}
\end{theorem}
\begin{proof}
The result follows from combining Lemma \ref{l3.5} and Theorem \ref{t4.1}.
\end{proof}

A few comments are in order at this point.

The sudden appearance of the exponential term $\exp(\tr(T_2(z)))$ in \eqref{4.30} and \eqref{4.31}, when compared to the one-dimensional case, is due to the necessary use of the modified determinant ${\det}_p(\cdot)$ in Theorems \ref{t4.1} and \ref{t4.2}.

The multi-dimensional extension \eqref{4.30} of
\eqref{1.16}, under the stronger hypothesis $V\in L^2(\Om; d^n x)$, $n=2,3$, first appeared in \cite{GLMZ05}. However, the present results in Theorem \ref{t4.2} go decidedly beyond those in \cite{GLMZ05} in the sense that the class of domains 
$\Omega$ permitted by Hypothesis \ref{h2.1} is greatly enlarged as compared to \cite{GLMZ05} and the conditions on $V$ satisfying Hypothesis \ref{h2.6} are nearly optimal by comparison with the Sobolev inequality (cf.\ Cheney \cite{Ch84}, Reed and Simon \cite[Sect.\ IX.4]{RS75}, Simon \cite[Sect.\ I.1]{Si71}). Moreover, the multi-dimensional extension \eqref{4.31} of \eqref{1.17} invoking Dirichlet-to-Neumann maps is a new result.

The principal reduction in Theorem \ref{t4.2} reduces (a ratio of) modified Fredholm determinants associated with operators in $L^2(\Om; d^n x)$ on the left-hand side of \eqref{4.30} to modified Fredholm determinants associated with operators in $L^2(\partial\Om; d^{n-1} \sigma)$ on the right-hand side of \eqref{4.30} and especially, in \eqref{4.31}. This is the analog of the reduction described in the one-dimensional context of Theorem \ref{t1.2}, where $\Om$ corresponds to the half-line $(0,\infty)$ and its boundary $\partial\Om$ thus corresponds to the one-point set $\{0\}$.

In the context of elliptic operators on smooth $k$-dimensional manifolds, the idea of reducing a ratio of zeta-function regularized determinants to a calculation over the
$(k-1)$-dimensional boundary has been studied by Forman \cite{Fo87}. He also pointed out that if the manifold consists of an interval, the special  case of a pair of boundary points then permits one to reduce the zeta-function regularized determinant to the determinant of a finite-dimensional matrix. The latter case is of course an analog of the one-dimensional Jost and Pais formula mentioned in the introduction (cf.\ Theorems \ref{t1.1} and  \ref{t1.2}). Since then, this topic has been further developed in various directions and we refer, for instance, to Burghelea, Friedlander, and Kappeler \cite {BFK91}, \cite {BFK92}, \cite {BFK93}, \cite{BFK95}, Carron \cite{Ca02}, Friedlander \cite{Fr05}, M\"uller \cite{Mu98}, Park and Wojciechowski \cite{PW05}, and the references therein.

\begin{remark}  \lb{c4.3} The following observation yields a simple
application of formula \eqref{4.30}. Since by the Birman--Schwinger principle (cf., e.g., the discussion in \cite[Sect.\ 3]{GLMZ05}), for any
$z\in\bbC\big\backslash \big(\si\big(H_{\Om}^D\big)\cup
\si\big(H_{0,\Om}^D\big) \cup \si\big(H_{0,\Om}^N\big)\big)$, one has 
$z\in\si\big(H_{\Om}^N\big)$ if and only if $\det{}_2\Big(I_{\Om}+ \ol{u\big(H_{0,\Om}^N -zI_{\Om}\big)^{-1}v}\,\Big)=0$, it follows from \eqref{4.30} that
\begin{align}
\begin{split}
&\text{for all } \, z\in\bbC\big\backslash \big(\si\big(H_{\Om}^D\big)\cup
\si\big(H_{0,\Om}^D\big) \cup \si\big(H_{0,\Om}^N\big)\big), \text{ one has }
z\in\si\big(H_{\Om}^N\big)  \\
&\quad \text{if and only if } \, \det{}_2\Big( I_{\dOm} -
\ol{\ga_N\big(H_{\Om}^D-zI_{\Om}\big)^{-1}V
\big[\ga_D\big(H_{0,\Om}^N-\ol{z}I_{\Om}\big)^{-1}\big]^*}\, \Big)=0.
\end{split}
\end{align}
One can also prove the following analog of \eqref{4.30}:
\begin{align}
&\frac{\det{}_2\Big(I_{\Om}+\ol{u\big(H_{0,\Om}^D-zI_{\Om}\big)^{-1}v}\,\Big)}
{\det{}_2\Big(I_{\Om}+\ol{u\big(H_{0,\Om}^N-zI_{\Om}\big)^{-1}v}\,\Big)} \no \\
&\quad = \det{}_2\Big(I_{\dOm} +
\ol{\ga_N\big(H_{0,\Om}^D-zI_{\Om}\big)^{-1}V
\big[\ga_D\big((H_{\Om}^N-z I_{\Om}\big)^{-1})^*\big]^*}\,\Big)
\lb{4.37} \\
&\quad\quad \times \exp\Big(-\tr\Big(\, 
\ol{\ga_N\big(H_{0,\Om}^D-zI_{\Om}\big)^{-1}V
\big(H_{\Om}^N-zI_{\Om}\big)^{-1}V
\big[\ga_D\big(H_{0,\Om}^N-\ol{z}I_{\Om}\big)^{-1}\big]^*}
\,\Big)\Big). \no
\end{align}
Then, proceeding as before, one obtains
\begin{align}
& \text{for all } \, z \in \bbC \big\backslash
\big(\si\big(H_{\Om}^N\big)\cup\si\big(H_{0,\Om}^N\big)\cup\si\big(H_{0,\Om}^D\big)\big),
\text{ one has } z\in\si\big(H_{\Om}^D\big) \\
&\quad \text{if and only if }\, \det{}_2\Big(I_{\dOm} +
\ol{\ga_N\big(H_{0,\Om}^D-zI_{\Om}\big)^{-1}V
\big[\ga_D\big(\big(H_{\Om}^N-z I_{\Om}\big)^{-1}\big)^*\big]^*} \,\Big)=0.  \no
\end{align}
\end{remark}

\noindent {\bf Acknowledgments.}
We are indebted to Yuri Latushkin and Konstantin A. Makarov for numerous
discussions on this topic. We also thank the referee for a careful reading of 
our manuscript and his constructive comments. 

Fritz Gesztesy would like to thank all 
organizers of the international conference on Operator Theory and Mathematical Physics (OTAMP), and especially, Pavel Kurasov, for their kind invitation, the stimulating atmosphere during the meeting, and the hospitality extended to him during his stay in Lund in June of 2006. He also gratefully acknowledges a research leave for the academic year 2005/06 granted by the Research Council and the Office of Research of the University of Missouri--Columbia. 


\end{document}